\documentclass[10pt,a4paper,reqno]{amsart}
\usepackage{epsfig}
\usepackage{color}
\usepackage{amssymb,amsmath,amsthm,amstext,amsfonts}
\usepackage{amssymb,amsmath,amsthm,amstext,amsfonts}
\usepackage{amsmath,amstext,amsthm,amsfonts}
\usepackage{color}
\usepackage[mathscr]{eucal}
\usepackage{dsfont}

\usepackage{psfrag}
\usepackage{url}
\usepackage{epstopdf}
\usepackage{epic,eepic}
\usepackage{upgreek}
\usepackage{amssymb}
\usepackage{mathrsfs}
\usepackage{verbatim}
\usepackage{mathrsfs}
\usepackage{amstext}
\usepackage{amsthm}
\usepackage{amssymb}
\usepackage{graphicx}

\theoremstyle{plain}
\newtheorem{theorem}{Theorem}[section]
\newtheorem{proposition}[theorem]{Proposition}

\newtheorem{corollary}[theorem]{Corollary}
\newtheorem{maintheorem}{Theorem}

\theoremstyle{definition}

\newtheorem{example}[theorem]{Example}
\newtheorem{definition}[theorem]{Definition}

\newcommand{\cF}{{\mathcal F}}

\title[Entropy points and applications for free semigroup actions]{Entropy points and applications for free semigroup actions}

\begin{document}

\author[F. Rodrigues]{Fagner B. Rodrigues*}
\thanks{*Corresponding author e-mail: fagnerbernardini@gmail.com}
\address{Departamento de Matem\'atica, Universidade Federal do Rio Grande do Sul, Brazil.}
\email{fagnerbernardini@gmail.com}

\author[T. Jacobus]{Thomas Jacobus}
\address{Departamento de Matem\'atica, Universidade Federal do Rio Grande do Sul, Brazil.}
\email{jacobus.math@gmail.com }

\author[M. Silva]{Marcus V. Silva}
\address{Departamento de Matem\'atica, Universidade Federal do Rio Grande do Sul, Brazil.}
\email{marcus423@gmail.com }

\keywords{Free semigroup action, topological entropy, entropy points,  skew-products.}

\subjclass[2000]{
Primary: 37B05, 
37B40 
Secondary: 37D20 
37D35; 
37C85  
}

\maketitle
\begin{abstract}
The aim of this manuscript is to study some local properties of the topological entropy of a free semigroup action. In order to do that we focus on the set of entropy points of a free semigroup action, show that this set carries the full entropy of the system (which, with respect to the chaocity of the system, gives a fundamental relevance to such set)  and obtain many interesting properties of such set. Our results are inspired by the ones presented in \cite{YZ}.
\end{abstract}

\setcounter{tocdepth}{1}

\section{Introduction}
The concept of entropy, introduced into the realm of dynamical systems more than fifty years
ago, has become an important ingredient in the characterization of the complexity of dynamical
systems. 

In the classical setting, $f:(X,d)\to (X,d)$ is a continuous map  acting on a compact metric space and the  topological entropy, denoted by 
$h_{\text{top}}(f)$, counts, in exponential scale, the number of distinguishable orbits. It
is   an important tool to describe the chaotic behaviour of $f$. In the end of the last century Bufetov \cite{Bufetov} presented 
a definition of topological entropy   of a free semigroup action $G$
generated by a finite set of continuous  maps $G_1=\{id,f_1,\dots,f_p\}$ acting on a compact metric space $X$ and the fixed random walk $\eta_p=\left(\frac1p,\dots,\frac1p\right)^\mathbb N$. Moreover, the author related it with the topological entropy of  the shift map $\sigma: \Sigma_p^+\to \Sigma_p^+$ and the induced continuous skew product $\cF_G:  \Sigma_p^+\times X\to \Sigma_p^+\times X$. The actions of semigroups and their dynamic aspects have been extensively studied in recent years, see for example \cite{Bis,BisII,BiW,BiU,BCMV,GLW,Ru,LMW2016,S00,XM}.
Recently, in \cite{CRV,CRVII,CRVIII}, a more general definition of topological entropy based in the one introduced in \cite{Bufetov} was exploited. In
that context the random walk considered  in $G$ is any probability measure in $\Sigma_p^+$. 

In \cite{YZ} the authors explored the concept of entropy point (point for which the topological entropy is positive for any neighborhood that contains it) for a single dynamic. Through the results presented in this work, it was possible to see the importance of such a notion for a better understanding of the local behavior of the dynamics, its topological entropy and also its metric entropy, since they showed that the support of ergodic measures is contained in the set of entropy points and that the metric entropy of an ergodic measure represents a lower bound for the topological  entropy of the Borelian sets, among other important results, as the ones about the entropy function (see Definition \ref{def:entropy-function}).

The main objective of this work is, through the definition of topological entropy of \cite{Bufetov}, to obtain important properties that help to describe the local behavior of the action and its topological entropy. Our results are motivated by those presented in \cite{YZ} and, as in \cite{YZ}, we show that the set of entropy points of the free semigroup action contains all the topological entropy of the system; given a measure of probability in the phase space, under certain conditions, the metric entropy of the action with respect to that measure represents a lower bound for the topological entropy of any Borelian set; there is always a closed enumerable set with total entropy, such set has at most one accumulation point and, in the case of the existence of a limit point, the entropy function evaluated at that point coincides with the entropy of the action.
\subsection{Setting}
Given a finite set of continuous maps $g_i:X \to X$, $i \in \mathcal{P}=\{1,2,\ldots,p\}$, $p\ge 1$, and the finitely generated semigroup $(G,\,\circ)$ with the finite set of generators $G_1=\{id, g_1, g_2, \dots, g_p\}$, we write
$G=\bigcup_{n\in \mathbb N_0} G_n$
where $G_0=\{id\}$ and $\underline g \in G_n$ if and only if $\underline g=g_{i_n} \dots g_{i_2} g_{i_1}$, with $g_{i_j} \in G_1$ (for notational simplicity's sake we will use $g_j \, g_i$ instead of the composition $g_j\,\circ\, g_i$). A semigroup can have multiple generating sets. We will assume that the generator set $G_1$ is minimal, meaning that no function $g_j$, for $j = 1, \ldots, p$, can be expressed as a composition of the remaining generators.
We shall consider different concatenations instead of the elements in $G$ they create.
One way to interpret this statement is to consider the itinerary map
$\iota : \mathbb{F}_p  \to  G $ given by
$\underline i=i_n \dots i_1  \mapsto  \underline g_{\underline i} := g_{i_n} \dots g_{i_1}$.
where $\mathbb{F}_p$ is the free semigroup with $p$ generators, and to regard concatenations on $G$ as images by $\iota$ of paths on $\mathbb{F}_p$.
Set $G_1^* =G_1 \setminus \{id\}$ and, for every $n\ge 1$, let $G_n^*$ denote the space of concatenations of $n$ elements in $G_1^*$. To summon each element $\underline{g}$ of $G^*_n$, we will write $|\underline{g}|=n$ instead of $\underline g\in G^*_n$. In $G$, one consider the semigroup operation of concatenation defined as usual: if $\underline{g}=g_{i_n} \dots g_{i_2} g_{i_1}$ and $\underline{h}=h_{i_m} \dots h_{i_2} h_{i_1}$, where $n=|\underline g|$ and $m=|\underline h|$, then
$\underline{g}\,\underline{h}=g_{i_n} \dots g_{i_2} g_{i_1} h_{i_m} \dots h_{i_2} h_{i_1} \in G_{m+n}^*.$
The finitely generated semigroup $G$ induces an \emph{action} in $X$, say
$$
\begin{array}{rccc}
\mathbb{S} : & G \times X & \to & X \\
	& (g,x) & \mapsto & g(x).
\end{array}
$$
We say that $\mathbb{S}$ is a \emph{semigroup action} if, for any $\underline{g},\,\underline{h}  \in G$ and every $x \in X$, we have $\mathbb{S}(\underline{g}\,\underline{h},x)=\mathbb{S}(\underline{g}, \mathbb{S}(\underline{h},x)).$ The action $\mathbb{S}$ is continuous if the map $\underline g : X \to X$ 
is continuous for any $\underline g \in G$.

\section{Preliminaries and main results}
Consider a finitely generated  free semigroup $(G,G_1)$
acting on a compact metric space $X$.
Let $K\subset X$ be a compact set. Given $\underline g= g_{i_n} \dots g_{i_1} \in G_n$, we say a set $E \subset K$ is \emph{$(\underline g,\varepsilon)$-separated set} if $d_{\underline g}(x_1,x_2) > \varepsilon$
 for any distinct $x_1,x_2\in E$. When no confusion is possible with the notation for the concatenation
 of semigroup elements, the maximum  cardinality of a
 $(\underline g,\varepsilon)$-separated sets of $K$ will
 be denoted by $s(K,\underline g, \varepsilon)$. We say that $F\subset K$ is a \emph{$(\underline g,\varepsilon)$-spanning set} if 
 given $x\in K$ there exists $y\in F$ so that if $d_{\underline g}(x,y) < \varepsilon$. When no confusion is possible with the notation for the concatenation
 of semigroup elements, the minimum   cardinality of a
 $(\underline g,\varepsilon)$-spanning  sets of $K$ will
 be denoted by $b(K,\underline g, \varepsilon)$.
 We now recall the notion of topological entropy introduced by
Bufetov~\cite{Bufetov}.


\begin{definition}\label{def:entropyB}
Given a compact set $K\subset X$,
we define
\begin{equation}\label{eq:entropiaB}
h_{top}(K,\mathbb S)
	=\lim_{\varepsilon\to 0} \limsup_{n\to\infty} \frac1n \log S_n(K,\mathbb S, \varepsilon),
\end{equation}
where
\begin{equation}\label{eq:Zn}
S_n(K,\mathbb S, \varepsilon)
	=\frac{1}{p^n}\sum_{{\underline g \in G_n^*}} s(K,\underline g,  \varepsilon),
\end{equation}
where the sum is taken over all concatenation $\underline g$ of $n$-elements of
$G_1 \setminus \{id \}$ and $p = |G_1 \setminus \{id \}|$. The
\emph{topological entropy $h_{top}(K,\mathbb S)$} is defined for $E=X$.
\\
\end{definition}

With Definition \ref{eq:entropiaB} we can can talk about entropy points:
\begin{quote}
(1) We say that $x_0\in X$ is an \emph{entropy point $x_0$} if for any closed
neighbourhood $K$ of $x_0$ we have
$$
h_{top}(K,\mathbb S)>0.
$$
\\
(2)  We say that $x_0\in X$ is a full \emph{entropy point $x_0$} if for any closed
neighbourhood $K$ of $x_0$ we have
$$
h_{top}(K,\mathbb S)=h_{top}(X,\mathbb S)
$$
holds.
\end{quote}
 Entropy points are those for which local neighborhoods reflect the complexity of the
entire dynamical system.

In \cite{RoVa1}, as in the case of a single dynamics, the authors proved that the  set of full entropy points is not empty. More precisely, it holds the following:
\begin{theorem}\label{thm1}
Let $\mathbb{S}:G\times X\to X$ be a finitely generated free semigroup action. Then
  $E^f_p(X,\mathbb{S})\not=\emptyset$.
\end{theorem}
Although not entirely related, the dynamics of semigroup actions has a strong connection with skew-products.
Indeed, if $X$ is a compact metric space 
and one considers a finite set of continuous maps $g_i:X \to X$, $i \in \{1,2,\ldots,p\}$, $p\ge 1$
consider the step skew-product
\begin{equation}\label{de.skew-product}
\begin{array}{rccc}
\mathcal{F}_G : & \Sigma_p^+  \times X & \to & \Sigma_p^+  \times X \\
	& (\omega,x) & \mapsto & (\sigma(\omega), g_{\omega_1}(x))
\end{array}
\end{equation}
where $\omega=(\omega_1,\omega_2, \dots)$ is an element of the full unilateral space of sequences
$\Sigma_p^+ = \{1,2,\ldots,p\}^\mathbb{N}$ and $\sigma$ denotes the shift map on $\Sigma_p^+$. With this notation
we will write $\cF_G^n(\omega,x) = (\sigma^n(\omega), f_\omega^n(x))$ for every $n\ge 1$.

\subsection{metric entropy and entropy points}
In what follows  we denote by $\mathcal{M}(X)$ the set of measures on $X$, by $\mathcal{M}_1(X)$
the set of probability measures on $X$ and for a given continuous map $f:X\to X$ we denote by $\mathcal{M}_{\text{inv}}(X)$ the set of 
$f$-invriant probability measures in $X$. In \cite{CRVIII} the authors defined the metric entropy of a free semigroup action $\mathbb S$ with respect $\nu\in \mathcal M_1(X)$ as 

\begin{align*}
   h_\nu(\mathbb{S})=\sup_{ \mu\in\Pi(\nu,\sigma)}h_\mu(\cF_G)-\log p,
\end{align*}
where $\Pi(\nu,\sigma)$ is the subset of $\mathcal M_{1}(\Sigma_p^+\times X)$ which are  $\cF_G$-invariant, $(\pi_{\Sigma_p^+})_\ast\mu$ is $\sigma$-invariant and $(\pi_X)_\ast\mu=\nu$. They proved that it satisfies a variational principle given by 
$$h_{\text{top}}(\mathbb{S}) = \sup_{\{\nu \, \in \, \mathcal{M}(X)\,\colon\, \Pi(\nu,\sigma) \neq \emptyset\}} \,\,h_\nu(\mathbb{S},\eta_{\underline a}).$$
Denote by $E_p(X, \mathbb S)$ the set of entropy points of $\mathbb S$, by $E^f_p(X, \mathbb S)$ the set of all full entropy points of $\mathbb S$. We also denote by $E_p(\Sigma_p^+\times X, \cF_G)$ the set of entropy points of $\cF_G$ and by $E^f_p(\Sigma_p^+\times X, \cF_G)$ the set of full entropy points of $\cF_G$.

Our first result shows that the metric entropy and the topological entropy of a free semigroup action are concentrated on the the set of entropy points. 

\begin{maintheorem}\label{thm2}
    Let $\mathbb{S}:G\times X\to X$ be a finitely generated free semigroup action. Then
  \begin{enumerate}
      \item[i.] Let $\nu\in \mathcal M_1(X)$ so that $\Pi(\sigma,\nu)_{erg}\not=\emptyset$, then $\text{supp}(\nu)\subset E_p(X,\mathbb{S})$;
      \item[ii.] $h_{top}(E_p(X,\mathbb{S}))=h_{top}(X,\mathbb{S})$.
  \end{enumerate}
  \end{maintheorem}
  As proved in \cite{YZ}, it is possible to get an upper bound for the metric of $\nu\in\mathcal M_1(X)$ for which $\Pi(\sigma,\nu)\not=\emptyset$
  in terms of the local information of the topological entropy.
  \begin{maintheorem}\label{thm3}
Let $\mathbb{S}:G\times X\to X$ be a finitely generated free semigroup action, $d$ a metric on $X$ $\nu\in M_1(X)$ so that $\Pi(\sigma,\nu)_{erg}\not=\emptyset$. Then 
\begin{align*}
    \liminf_{\varepsilon\to0}\{S_d(K,\mathbb S,\varepsilon):K\in \mathcal B_X \text{ with }\nu(K)>0\}\geq h_\nu(\mathbb S).
\end{align*}
In particular, for any $K\in\mathcal B_X$, $h_{top}(K,\mathbb S)\geq h_\nu(\mathbb S)$.
\end{maintheorem}

\subsection{Entropy function}
For each $\varepsilon>0$ and $x\in X$, define 
\[
h_d(x,\varepsilon)=\inf\{B_d(K,\mathbb S,\varepsilon): K \text{ is compact neighbourhood of }x\},
\]
where the subscribed $d$ is to emphasize the metric $d$ and 
\[
B_d(K,\mathbb S,\varepsilon)=\limsup_{n\to\infty}\frac{1}{n}\log\left(\frac{1}{p^n}\sum_{\underline  g\in G_n^*}b_d(K,\underline g,\varepsilon)\right),
\]
and $b_d\left(K,\underline g,\varepsilon)\right)$ denotes the minimum cardinality of a $(\underline g,\varepsilon)$-spanning set in the metric $d$. 
As $h_d(x,\varepsilon)$ increases as $\varepsilon$ decreases to zero, it is well defined the following
\begin{align}\label{eq:entropy-function}
h_d(x)=\lim_{\varepsilon\to0^+}h_d(x,\varepsilon)
\end{align}
and it is less or equal to $h_{top}(X,\mathbb S)$. In fact, it depends only on the topology of $X$ and
we can denote by $h_{top}(x)$.

\begin{definition}\label{def:entropy-function}
Let $\mathbb S:G\times X\to X$ be a continuous finitely generated free semigroup
action. The function $h_{top}:X\to [0,h_{top}(X,\mathbb S)]$, $x\mapsto h_{top}(x)$
is called the entropy function of $\mathbb S$.
\end{definition}

Since $B_d(K,\mathbb S,\varepsilon)\leq S_d(K,\mathbb S,\varepsilon)\leq B_d(K,\mathbb S,\varepsilon\slash 2)$, we have
\[
h_{top}(x)=\lim_{\varepsilon\to0}\inf\{S_d(K,\mathbb S,\varepsilon): K \text{ is compact neighbourhood of }x\}.
\]

The next theorem gives a lower bound, depending on the metric entropy, for the entropy function in each point. Moreover, it is proved 
that, as one may hope, the suppremum of the entropy function is given by the topological entropy.
\begin{maintheorem}\label{thm4}
Let $\mathbb S:G\times X\to X$ be a continuous finitely generated free semigroup
action.  
\begin{enumerate}
\item[i.] $h_{top}(x)\geq \sup\{h_\nu(\mathbb S):\nu\in\mathcal M_1(X),\;\Pi(\sigma,\nu)_{erg}\not=\emptyset\text{ and }x\in supp(\nu) \}$.
  \item [ii.]  For each $\varepsilon>0$, $h_d(\cdot,\varepsilon)$ is upper semi continuous, consequently,
 $h$ is  Borel measurable.
  \item[iii.] If $K\subset X$ is closed, then $\sup_{x\in K}h_{top}(x)\geq h_{top}(K,\mathbb S)$. In particular, 
  \[\sup_{x\in X} h_{top}(x)=\ h_{top}(X,\mathbb S).\]
  \item[iv.] If $G_1$ is a finite set of homeomorphisms on X then $h$ is $\mathbb S$-invariant and for $\nu\in \mathcal M_1(X)$ satisfying 
  $\Pi(\sigma,\nu)_{erg}\not=\emptyset$, 
  $$
  \int_X h_{top}(x)\;d\nu\geq h_\nu(\mathbb S).
  $$
\end{enumerate}
 \end{maintheorem}
 Our last result shows that the topological entropy of a free semigroup action may be computed in terms 
 of countable sets.
 \begin{maintheorem}\label{thm5}
 Let $\mathbb S:G\times X\to X$ be a continuous finitely generated free semigroup
action. Then there exists a countable closed subset
$K \subset X$ such that $h_{top}(K,\mathbb S) = h_{top}(X,\mathbb S)$. Moreover, $K$ can be chosen such that the set of
limit points of $ K$ has at most one limit point, and $K$ has a unique limit point if and only if
there is $x \in X$ with $h_{top}(x) = h_{top}(X,\mathbb S)$.
 \end{maintheorem}

\section{proofs}

Before we prove the main results we present some important observations on the characterization of the set of entropy points of a free semigroup action. 
\begin{proposition}\label{lemma:existence}
Let $\omega=\omega_1\omega_2\dots\in \Sigma_p^+$. Then
$h_{top}([\omega_1\dots\omega_\ell]\times K,\cF_G)= h_{top}(K,\mathbb S)+\log p$.
\end{proposition}

\begin{proof}
Let $n$ be a positive integer and consider  $N=p^n$. There are $N$ distinct words of
length $n$ in $\mathbb{F}_p^+$. Denote these words by $\theta^1,\dots,\theta^N$. Let
$(\theta(i))_{i=1}^N\subset \Sigma_p^+$ be a sequence such that $\theta(i)|_{[1,n]}=\theta^i$.
We notice that, for $0<\varepsilon<\frac{1}{2}$, the sequence $(\omega_1\dots\omega_\ell\theta(i))_{i=1}^N\subset [\omega_1\dots\omega_\ell]$
is $(n,\varepsilon,\sigma)$-separated subset of $[\omega_1\dots\omega_\ell]$.

Let $\theta^i=\theta_1^i\dots\theta_n^{i}$ and $\underline g^{(i)}=g_{\theta^i_n}\dots g_{\theta^i_1}g_{\omega_\ell}\dots g_{\omega_1}$. Set $Z_i=Z(\underline g^{(i)},\varepsilon,K)$, where $\underline g^{(i)}$ is the element in $G$ given by the c and let the points $x_1^i,\dots,x_{Z_i}^i$ form a $(\underline g^{(i)}, \varepsilon)$-separated subset
of $K$. Then the points
$$
(\omega_1\dots\omega_\ell\theta(i), x_j^i)\in [\omega_1\dots\omega_\ell]\times K,\ \ \ \ i=1,\dots,N \ \ \ j=1,\dots, Z_i,
$$
form a $(n,\varepsilon,\cF_G)$-separated subset of $[\omega_1\dots\omega_\ell]\times K$. It implies that
$$
S([\omega_1\dots\omega_\ell]\times K,\cF_G,n,\varepsilon)\geq \sum_{i=1}^{N}s(K,\underline g^{(i)},\varepsilon),
$$
and then $h_{top}([\omega_1\dots\omega_\ell]\times K,\cF_G)\geq h_{top}(K,\mathbb{S})+\log p$.

For the converse inequality, consider   $\varepsilon>0$ and take $C(\varepsilon)$ an arbitrary positive integer
such that $p^{-C(\varepsilon)}<\frac{\varepsilon}{100}$. We notice that there exist $N=p^{n+2C(\varepsilon)}$
distinct words of length $n+2C(\varepsilon)+\ell$. Denote these words by $\omega^1,\dots,\omega^N$.

Let $(\omega(i))_{i=1}^N\subset\Sigma_p^+$ be an arbitrary sequence satisfying
$\omega(i)|_{[1,n+C(\varepsilon)+\ell]}=\omega^i$. This sequence forms a $(n,\varepsilon,\sigma)$-spanning set
of $[\omega_1\dots\omega_\ell]$. Denote $\theta^i=\omega(i)|_{[0,n+\ell]}$, by $\underline g^{(i)}$ the concatenation associated to 
$\theta^i$ and $B_i=B(\underline g^{(i)},\varepsilon,K)$
and assume that the points $x_1^i,\dots,x_{B_i}^i$ form a $(\underline g^{(i)},\varepsilon)$-spanning subset of $K$. It follows that
the points $(\omega(i),x_j^i)$ with $i=1,\dots, N$ and $j=1,\dots, B_i$, form a $(\underline g^{(i)},n,\varepsilon)$-spanning set
of $\ K$. In that case, there exists a positive constant $T=T(\varepsilon)$ depending only
on $\varepsilon $ so that
$$
B([\omega_1\dots\omega_\ell]\times K ,\cF_G,n+\ell,\varepsilon)\leq \sum_{i=1}^{n+2C(\varepsilon)+\ell}b(K,\mathbb S,\underline g^{(i)},\varepsilon)
\leq T(\varepsilon)\sum_{|\underline g|=n+\ell}b(K,\mathbb S,\underline g,\varepsilon).
$$
From the last inequalities we obtain 
\begin{align*}
h_{top}([\omega_1\dots\omega_\ell]\times K,\cF_G)\\
&\leq h_{top}(K,\mathbb S)+\log p.
\end{align*}

It concludes the proof.

\end{proof}
As an immediate consequence of Proposition \ref{lemma:existence} we have the following.
\begin{corollary}
Under the above notations we have that
$E_p(\Sigma_p^+\times X, \cF_G)=\Sigma_p^+\times X$.
\end{corollary}

\begin{proof}
Given $(\omega,x)\in\Sigma_p^+\times X$ and $F$ closed neighbourhood containing $(\omega,x)$, by the definition of the
 product topology of $\Sigma_p^+\times X$, there exist closed subsets  $[\omega_1\dots\omega_\ell]\subset \Sigma_p^+$ and $K\subset X$, so that
 $[\omega_1\dots\omega_\ell]\times K\subset F$. As $h_{top}(F,\cF)\geq h_{top}([\omega_1\dots\omega_\ell]\times K)$ and, by Proposition \ref{lemma:existence}, $h_{top}([\omega_1\dots\omega_\ell]\times K)\geq \log p$.

\end{proof}
If one denotes by $\pi_X$ the canonical projection of $\Sigma_p^+\times X$ on the first coordinate, the next corollary guarantees that
if $(\omega,x)\in E^f_p(\Sigma_p^+\times X, \cF_G)$ then  $\pi_X(\omega,x)\in E^f_p(X, \mathbb S)$.

\begin{corollary}\label{lemma2}
Let $(\omega,x)\in \Sigma_p^+\times X$ be a full entropy point for $\cF_G$. Then $x\in  X$ is  a full entropy point for $\mathbb S$.
\end{corollary}

\begin{proof}
Assume that $\omega=\omega_1\omega_2\dots$ and consider $[\omega_1\dots\omega_\ell]\times K$ compact
neighbourhood of $(\omega,x)$. By Proposition \ref{lemma:existence}
$$
h_{top}(X, \mathbb S)+\log p=h_{top}(\Sigma^+_p\times X,\cF_G)=
h_{top}([\omega_1\dots\omega_\ell]\times K,\cF_G)= h_{top}(K,\mathbb S)+\log p.
$$
\end{proof}

Another important consequence of Proposition \ref{lemma:existence} is the following,
\begin{corollary}\label{prop:1}
 Let $\mathbb{S}:G\times X\to X$ be a finitely generated free semigroup action. Then
  $E^f_p(X,\mathbb{S})\not=\emptyset$.
\end{corollary}
  


The fact of the semigroup action having positive entropy when restricted to a subset of $X$ is related presence of entropy points in such subset, as
the next shows.
\begin{proposition}\label{proposition-closed-set}
Let $\mathbb{S}:G\times X\to X$ be a finitely generated free semigroup action and $K$ be a closed subset of $X$. 
  \begin{enumerate}
      \item[i.] If $h_{top}(K,\mathbb{S})>0$, then $K\cap E_p(X,\mathbb{S})\not=\emptyset$;
      \item[ii.]  If  $h_{top}(K,\mathbb{S})=h_{top}(X,\mathbb{S})$, then $K\cap E^f_p(X,\mathbb{S})\not=\emptyset$.
      \end{enumerate}
\end{proposition}
\begin{proof}
We start  proving (i). Cover $K$ by finitely many closed balls $B_1^{1},\dots,B_{\ell_1}^1$ with diameter at most 1. We know that
$h_{top}(K,\mathbb S)=\max_{j}h_{top}(B_{i_j},\mathbb S)$ and it implies the existence of $j_1$ so that $h_{top}(K,\mathbb S)=h_{top}(B_{j_1},\mathbb S)$. Cover $B_{j_1}$ by by finitely many closed balls $B_1^{2},\dots,B_{\ell_2}^2$ with diameter at most $\frac12$. By the same reasoning, there exists $j_2$ for which $h_{top}(K,\mathbb S)=h_{top}(B^1_{j_1},\mathbb S)=h_{top}(B^2_{j_2},\mathbb S)$. By induction, for
each $n \geq 2$, there is a closed ball $B^n_{j_n}\subset B^n_{j_{n-1}}$ with diameter at most $\frac1n $
such that $h_{top}(K,\mathbb S)=h_{top}(B^n_{j_n},\mathbb S)$. Set $x$ the unique point in the intersection of the closed balls $B^n_{j_n}$. By definition, $x$ is an entropy point for $\mathbb S$.

To prove (ii) we notice that, since $h_{top}(K,\mathbb{S})=h_{top}(X,\mathbb{S})$, then 
$h_{top}(\Sigma_p^+\times K,\cF_G)=h_{top}(\Sigma_p^+\times X,\cF_G)$. By \cite[Theorem 3.5]{YZ}  there exists $(\omega,x)\in \Sigma_p^+\times K\cap E^f_p(\Sigma_p^+\times X,\cF_G)$. By Corollary \ref{lemma2} we have that $\pi_X(\omega,x)\in K\cap  E^f_p(X,\mathbb{S})$, and it proves (ii).
\end{proof}

\subsection{Proof of Theorem \ref{thm2}}
Before we start the proof we recall the definition of Katok's entropy of a probability measure in $X$, which was extended 
for the semigroup setting in \cite{CRVIII}.
\subsubsection{Katok's entropy}

\begin{definition}\label{def:metric-entropy-5}
Given  a Borel probability measure $\nu$ on $X$, $\delta\in (0,1)$ and $\varepsilon>0$, define
\begin{equation}
h^{K}_{\nu}(\mathbb{S},\varepsilon,\delta) 
	= \limsup_{n\to\infty} \frac1n \log S_\nu(X,\mathbb S,n, \varepsilon,\delta),
\end{equation}
where
\begin{equation}\label{eq:Zn}
S_\nu(X,\mathbb S, n, \varepsilon,\delta)
	=\frac{1}{p^n}\sum_{{\underline g \in G_n^*}} s_\nu(\underline g,\varepsilon,\delta),
\end{equation}
and for $\underline g\in G_n^*$
$$s_\nu(\underline g, \varepsilon,\delta) = \,\inf_{\{E \,\subseteq \,X\,\colon\, \nu(E) \,> \,1-\delta\}}\,\, s(E,\underline g,  \varepsilon)$$
and $s(\underline g,  \varepsilon,E)$ denotes the maximal cardinality of the $(\underline g,  \varepsilon)-$separated subsets of $E$.
\end{definition}
 The \emph{entropy of the semigroup action $\mathbb{S}$ with respect to $\nu$} is defined by
\begin{equation}
h^{K}_{\nu}(\mathbb{S}) = \lim_{\delta \to 0} \,\,\lim_{\varepsilon \to 0}\,\, h^{K}_{\nu}(\mathbb{S},\varepsilon,\delta) 
\end{equation}

Observe that the previous limit is well defined due to the monotonicity of the function
$$(\varepsilon,\delta) \mapsto \frac1n \log S_\nu(X,\mathbb S, n, \varepsilon,\delta)$$
on the unknowns $\varepsilon$ and $\delta$. Moreover, if the set of generators is $G_1=\{Id,f\}$, we recover the notion proposed by Katok for a single dynamics $f$.

By \cite[Theorem C]{CRVIII} we have that 
$$
h_{top}(X,\mathbb S)=\sup_{\nu\in\mathcal M_1(X)}h^K_\nu(\mathbb S).
$$

Let us proceed to the proof of Theorem \ref{thm2}.
We recall that for $\nu\in\mathcal{M}_1(X)$, we denote the set of ergodic measures in $\Pi(\sigma,\nu)$ by $\Pi(\sigma,\nu)_{erg}$.

 Since $h_\nu(\mathbb S)>0$, by the definition of the metric entropy, we have that there exists $\mu\in\Pi(\sigma,\nu)_{erg}$ 
 for which $h_\mu(\cF_G)>\log p$. If $x\in \text{supp}(\nu)$, for any closed neighbourhood  $N_x$ of $x$, $\nu(N_x)=\mu(\Sigma_p^+\times  N_x)>0$. By \cite[Theorem 3.7]{YZ} we have that
 \begin{align*}
     h_{top}(N_x,\mathbb S)+\log p=h_{top}(\Sigma_p^+\times  N_x,\cF_G)\geq h_\mu(\cF_G)\log p,
 \end{align*}
which implies that $h_{top}(N_x,\mathbb S)>0$ and then, $x\in E_p(X,\mathbb S)$, concluding the proof item (i).

 For (ii)  we notice that
\begin{align*}
    h_{top}(X,\mathbb S)=\sup_{\{\nu\in\mathcal M_1(X):\Pi(\sigma,\nu)_{erg}\not=\emptyset\}}h_\nu(\mathbb S).
\end{align*}
By (i) we have that the suppremum is taken on the probability measures on $ X$ with support contained in $ E_p(X,\mathbb S)$. 
By \cite[Theorem C]{CRVIII} we have that  
$$
h_\nu(\mathbb S)\leq h_\nu^K(\mathbb S),
$$
and it implies that $h_{top}(X,\mathbb S)=h_{top}(E_p(X,\mathbb S),\mathbb S)$, since $\text{supp}(\nu)\subset E_p(X,\mathbb S)$.


\subsection{Proof of Theorem \ref{thm3}}
Let $0<\delta<1$ fixed. Let $\varepsilon>0$ and $K\in\mathcal B_X$ satisfy $\nu(K)>0$. As $\Pi(\sigma,\nu)_{erg}\not=\emptyset$, for $\mu\in\Pi(\sigma,\nu)_{erg}$, by the ergodicity of $\mu$, there exists $m(K)\in \mathbb N$  such that $\mu\left(\bigcup_{i=0}^{m(K)}\cF_G^i(\Sigma_p^+\times K)\right)\geq 1-\delta$. 

\textbf{Claim.} If $K^{(m)}:=\pi_X\left(\bigcup_{i=0}^{m(K)}\cF_G^i(\Sigma_p^+\times K)\right)$, then 
$S_d(K,\mathbb S,\varepsilon)=S_d(K^{(m)},\mathbb S,\varepsilon)$.

\noindent \textit{Proof of the Claim.} First of all we notice that $\nu(K^{(m)})\geq 1-\delta$. Fix $n\in \mathbb N$ and take $\underline g\in G_n^*$. Let $E_{\underline g}\subset K^{(m)}$ be a $(\underline g,\varepsilon)$-separated set.  In particular there exists $0\leq i_{0}(n)\leq m$
so that 
\begin{align}\label{eq:sep-cote}
    \left|E_{\underline g}\cap \pi_X(\cF_G(\Sigma_p^+\times K))\right|\geq\frac{|E_{\underline{g}}|}{m+1}.
\end{align}
If $\underline g=g_{j_n}\dots g_{j_1}$, we have that $\underline g_{i_0(n)}^{-1}(E)\cap  K$ is a $(\underline{g},\varepsilon)$-separated set in $K$. Taking $\tilde{\underline{g}}=\underline{g}\;\underline{ g}_{i_0(n)}$, we have that $ \underline{ g}_{i_0(n)}^{-1}(E)$ is $(\tilde{\underline{g}},\varepsilon)$-separated, with $\tilde{\underline{g}}\in G_{n+m}$. By \eqref{eq:sep-cote} 
\begin{align*}
    s(K,\tilde{\underline{g}},\varepsilon)\geq \frac{1}{m+1}s(K^{(m)},\underline g,\varepsilon)\geq\frac{1}{m+1} s( K,\underline g,\varepsilon).
\end{align*}
Denote by $\tilde G_{m+n}^*$ the subset of $\tilde{\underline{g}}$ obtained in the previous construction. It shows that 
\begin{align*}
    S_{m+n}(K,\mathbb S,\varepsilon)
    &=\frac{1}{p^{m+n}}\sum_{\underline g\in G_{m+n}^*}s(K,\underline{g},\varepsilon)\\
    &\geq\frac{1}{p^{m+n}} \frac{1}{m+1}\sum_{\tilde{\underline{g}}\in \tilde G_{m+n}^*}s(K,\tilde{\underline{g}},\varepsilon)\\
    &\geq \frac{1}{p^{m+n}}\frac{1}{m+1}\sum_{\underline g\in G_n^*}s(K^{(m)},\underline g,\varepsilon)\\
    &\geq \frac{1}{p^{m+n}}\frac{1}{m+1}\sum_{\underline g\in G_n^*}s(K,\underline g,\varepsilon),
\end{align*}
and it implies
$$S_{m+n}(K,\mathbb S,\varepsilon)\geq \frac{1}{p^{m}} \frac{1}{m+1} S_n(K^{(m)},\mathbb S,\varepsilon)\geq\frac{1}{p^{m}} \frac{1}{m+1} S_n(K,\mathbb S,\varepsilon),$$
which guarantees
$S(K,\mathbb S,\varepsilon)=S_d(K^{(m)},\mathbb S,\varepsilon)$, and this proves the claim.
Let us proceed to finalize the proof of the theorem. As, for any $\varepsilon>0$, $n\in\mathbb N$ and  $R\subset X$ 
$$
B_n(R,\mathbb S,\varepsilon)\leq S_n(R,\mathbb S,\varepsilon),
$$
 we have 
\begin{align*}
    &\inf\{S_d(K,\mathbb S,\varepsilon):K\in \mathcal B_X \text{ with }\nu(K)>0\}\\
    &=\inf\{S_d(K^{(m)},\mathbb S,\varepsilon):K\in \mathcal B_X \text{ with }\nu(K)>0\}\\
    &\geq \inf\left\{\limsup_{n\to\infty}\frac{1}{n}\log\frac{1}{p^n}B(K^{(m)},\mathbb S,\varepsilon):K\in \mathcal B_X \text{ with }\nu(K)>0\right\}\\
    &\geq \limsup_{n\to\infty}\frac{1}{n}\inf\left\{\log\frac{1}{p^n}B_n(K^{(m)},\mathbb S,\varepsilon):K\in \mathcal B_X \text{ with }\nu(K)>0\right\}\\
    &\geq \limsup_{n\to\infty}\frac{1}{n}\log\frac{1}{p^n}B_\nu(\mathbb S,n,\varepsilon).
\end{align*}
Letting $\varepsilon\to0$ and noticing that 

$$
\lim_{\varepsilon\to0}\limsup_{n\to\infty}\frac{1}{n}\log\frac{1}{p^n}B_\nu(\mathbb S,n,\varepsilon)\geq h_{\nu}(\mathbb S),
$$
we conclude the proof.
\subsection{Proof of Theorem \ref{thm4}}
%

%
We notice that (i) is an immediate consequence of Theorem \ref{thm3}.

For (ii), let $\varepsilon>0$. For $r\in\mathbb R$, if $h_{top}(x_0,\varepsilon)<r$, then $B(K,\mathbb S,\varepsilon)<r$, for some closed 
neighbourhood $K$ of $x_0$. In particular, $h_{top}(x,\varepsilon)<r$ for each $x\in K$. So, $h_{top}(\cdot,\varepsilon)$
is upper semi continuous. As we can obtain $h$  as the limit of $\{h_{top}(x,\frac1n)\}_{n\in\mathbb N}$, sequence of upper semi continuous maps, it is a Borel map.

(iii). Cover $K$ by finitely many closed balls $B_1^{1},\dots,B_{\ell_1}^1$ with diameter at most 1. We know that
$h_{top}(K,\mathbb S)=\max_{j}h_{top}(B_{i_j},\mathbb S)$ and it implies the existence of $j_1$ so that $B_d(K,\mathbb S,\varepsilon)=B_d(B_{j_1}\cap K,\mathbb S,\varepsilon)$. Cover $B_{j_1}$  by finitely many closed balls $B_1^{2},\dots,B_{\ell_2}^2$ with diameter at most $\frac12$. By the same reasoning, there exists $j_2$ for which $B_d(K,\mathbb S,\varepsilon)=B_d(B_{j_2}\cap K,\mathbb S,\varepsilon)$. By induction, for
each $n \geq 2$, there is a closed ball $B^n_{j_n}\subset B^n_{j_{n-1}}$ with diameter at most $\frac1n $
such that $B_d(K,\mathbb S,\varepsilon)=B_d(B_{j_n}\cap K,\mathbb S,\varepsilon)$. Set $x_0$ the unique point in the intersection of the closed balls $B^n_{j_n}$, which belongs to $K$, since $K$ is closed. If $K'$ is any closed neighbourhood of $x_0$, for $n\in \mathbb N$ sufficiently large, $B_{j_n}\cap K\subset K'$, which implies
$$
B_d(K',\mathbb S,\varepsilon)\geq B_d(B_{j_n}\cap K,\mathbb S,\varepsilon)=B_d(B_{j_n}\cap K,\mathbb S,\varepsilon)
$$
and 
$$
h_{top}(x_0)\geq h_d(x_0,\varepsilon)\geq  B_d(B_{j_n}\cap K,\mathbb S,\varepsilon).
$$
Hence, letting $\varepsilon\to0^+$,
$$
\sup_{x\in K}h_{top}(x)\geq h_{top}(K,\mathbb S).
$$
For the last part of Theorem \ref{thm4} we need of the following proposition, which is a consequence the definition of the entropy function of $\mathbb S$ and $\cF_G$.

\begin{proposition}\label{prop3}
 Let $\mathbb{S}:G\times X\to X$ be a finitely generated free semigroup action. For $x\in X$,
 $$
 \sup_{\omega\in\Sigma_p^+}h_{D\times d}((\omega,x),)\leq h_d(x)+\log p.
 $$
\end{proposition}
\begin{proof}
Fix $\varepsilon>0$. Let $\omega\in \Sigma_p^+$ and $V\subset \Sigma_p^+\times X $ be a closed neighbourhood of $(\omega,x)$. 
Since we consider the product topology in $\Sigma_p^+\times X$, there exist $\Sigma\subset \Sigma_p^+$ and $Z\subset X$ so that
$(\omega,x)\in \Sigma\times Z\subset \Sigma_p^+\times X$. Let $Z=Z(\varepsilon)$ so that, for $\delta>0$,
$ h_d(x,\varepsilon)\leq B_{d}( Z,\mathbb S,x,\varepsilon)\leq h_d(x,\varepsilon)+\delta$.
It follows that 
\begin{align*}
h_{D\times d}((\omega,x),\varepsilon)
&\leq B_{D\times d}(\Sigma\times Z,\mathcal F_G,(\omega,x),\varepsilon)\\
&\leq B_{D\times d}(\Sigma_p^+\times Z,\cF_G,(\omega,x),\varepsilon)\\
&\leq B_{d}( Z,\mathbb S,x,\varepsilon)+\log p\\
&\leq h_d(x,\varepsilon)+\log p+\delta.
\end{align*}
As $\delta$ was taken arbitrary, $h_{D\times d}((\omega,x),\varepsilon)\leq  h_d(x,\varepsilon)+\log p$, which implies 
$$\sup_{\omega\in\Sigma_p^+}h_{D\times d}((\omega,x),)\leq h_d(x)+\log p
$$
and finishes the proof.
\end{proof}
Assuming that $G_1$ is a finite set of homeomorphisms of $X$
we have that $\mathcal F_G$ is a homeomorphism. By \cite[Theorem 4.5]{YZ}, $h_{top}(\cdot,\cdot)$ is $\cF_G$-invariant and if $\alpha\in \mathcal M_{erg}(\Sigma_p^+\times X)$
then 
$$
\int_{\Sigma_p^+\times X}h_{top}(\omega,x)d\alpha\geq h_\alpha(\cF_G),
$$
For $\nu=(\pi_X)_\ast\alpha$ we obtain, by Proposition \ref{prop3},
\begin{align*}
    h_\nu(\mathbb S)&=\sup_{\mu\in\Pi(\sigma,\nu)_{erg}}h_{\mu}-\log p\\
                    &\leq \sup_{\mu\in\Pi(\sigma,\nu)_{erg}}\int_{\Sigma_p^+\times X}h_{top}(\omega,x)d\mu-\log p\\
                    &\leq \sup_{\mu\in\Pi(\sigma,\nu)_{erg}}\int_{\Sigma_p^+\times X}h_{top}(x)d\mu-\log p\\
                    &=\int_{\Sigma_p^+\times X}h_{top}(x)d\nu.
\end{align*}

\subsection{Proof of Theorem \ref{thm5}}
We will get the proof of the theorem as a consequence of the following proposition.

\begin{proposition}\label{prop4}
  Let $\mathbb{S}:G\times X\to X$ be a finitely generated free semigroup action and $h $ its entropy function.
  \begin{enumerate}
      \item[i.] If $K$ is a  countable closed subset with a unique limit point $x_0$,  then $h_{top}(x_0)\geq h_{top}(K,\mathbb S)$.
      \item[ii.] Let $x_0\in X$. Then there exists a countable closed subset $K \subset X$ such
that $x_0\in K$ is its unique limit point in $X$ and $h_{top}(x_0)= h_{top}(K,\mathbb S)$.
  \end{enumerate}
\end{proposition}
\begin{proof}
Let $d$ be a metric on $X$. To simplify the notation we drop the subscribed $d$ in the definition of $h$. We start with (i). Fix $\varepsilon>0$. By hypotheses, for any closed neighbourhood $Z$ of $x_0$,
$K\backslash Z$ is a finite set and it implies $B(Z,\mathbb S,\varepsilon)\geq B(K,\mathbb S,\varepsilon)$. By the definition of $h$ we obtain
$h_{top}(x_0)\geq h_d(x_0,\varepsilon)\geq B(K,\mathbb S,\varepsilon)$. Letting $\varepsilon\to0$ we get the conclusion.

For (ii) we assume $h_{top}(x_0)<\infty$.  For each $\in\mathbb N$ we set 
$$
K_n=\left\{x\in X: d(x,x_0)\leq \frac{1}{n}\right\}.
$$
For each $m\in\mathbb N$ choose $\varepsilon_m$ such that
$$
h_{top}(x_0)-\frac{1}{m}<\inf_{n\in\mathbb N}S(K_n,\mathbb S,\varepsilon_m)=\inf_{n\in\mathbb N}\limsup_{k\to\infty}\frac{\log S_k(K_n,\mathbb S,\varepsilon_m)}{k}.
$$
So, there exists an increasing sequence  $\{k_{n,m}\}_{n\in\mathbb N}\subset\mathbb{N}$ such that for all $n\in\mathbb N$
$$
\frac{1}{p^{k_{n,m}}}\sum_{\underline g\in G^\ast_{k_{n,m}}}S(K_n,\underline g,\varepsilon_m)=S_{k_{n,m}}(K_n,\mathbb S,\varepsilon_m)\geq e^{k_{n,m}\left(h_{top}(x_0)-\frac{1}{m}\right)}.
$$
Now take $\underline g\in G^\ast_{k_{n,m}}$ and denote by $E_{\underline g,(m,n)}$  a $(\underline g,\varepsilon_m)$-separated set in $K_n$
of maximum cardinality. Define 
$$E_{m,n}=\bigcup_{\underline g\in G^\ast_{k_{n,m}}}E_{\underline g,(m,n)}\cup \{x_0\} \text{ and }K=\bigcup_{m\geq 1}\bigcup_{n\geq m}E_{n,m}.$$
If $V$ is a neighbourhood of $x_0$, by the definition of the sets $K_n$, we have that $K_{n_0}\subset V$ for $n_0$ large enough. So,
$$
K\backslash V\subset K\backslash K_{n_0}\subset \bigcup_{m= 1}^{n_0-1}\bigcup_{n=m}^{n_0-1}E_{n,m}.
$$
As each $E_{n,m}$ is a finite set, $K\backslash V$ is finite and it guarantees that $x_0$ is the unique possible limit point of $K$ in $X$.

Let us prove that $h_{top}(x_0)=h_{top}(K,\mathbb S)$. To do that fix $m\in\mathbb N$. For $n\geq m$, take $\underline g \in G^\ast_{k_{n,m}}$,
and notice that $E_{\underline g,(m,n)}$ is a $(\underline g,\varepsilon_m)$-separated set in $K$. It gives
\begin{align*}
    \frac{1}{p^{k_{n,m}}}\sum_{\underline g\in G^\ast_{k_{n,m}}}s(K,\underline g,\varepsilon_m)\geq \frac{1}{p^{k_{n,m}}} \sum_{\underline g\in G^\ast_{k_{n,m}}}|E_{\underline g,(m,n)}|\geq e^{k_{n,m}\left(h_{top}(x_0)-\frac{1}{m}\right)}.
\end{align*}
In such case we obtain
\begin{align*}
    h_{top}(K,\mathbb S)\geq S(K,\mathbb S,\varepsilon_m)
    &=\limsup_{k\to \infty}\frac{\log S_{k}(K,\mathbb S,\varepsilon_m)}{k}\\
    &\geq \limsup_{n\to \infty}\frac{\log S_{k_{n,m}}(K,\mathbb S,\varepsilon_m)}{k_{n,m}}\\
    &\geq\limsup_{n\to \infty}\frac{\log e^{k_{n,m}\left(h_{top}(x_0)-\frac{1}{m}\right)}}{k_{n,m}}\\
    &=h_{top}(x_0)-\frac{1}{m}.
     \end{align*}
     As the equality holds for all $m\in \mathbb N$ we obtain $h_{top}(K,\mathbb S)\geq h_{top}(x_0)$. By (i) we get the desired equality in (ii).

\end{proof}
Let us proceed to the proof of Theorem \ref{thm5}. For any $\varepsilon>0$ and any $\varepsilon>0$ we consider $B_\varepsilon(x)$ the open 
ball of radius $\varepsilon>0$ and center $x$. As $X$ is compact, by Theorem \ref{thm4}, there exists $\{x_n\}_{n\in\mathbb N}\subset X$ such that 
$$
\lim_{n\to\infty}x_n=x_0 \text{ and }\lim_{n\to\infty}h_{top}(x_n)=h_{top}(X,\mathbb S).
$$
Let $\{r_n\}_{n\in\mathbb N}$ be any given sequence of positive real numbers which converges to 0.
Applying  Proposition \ref{prop4}, given $n\in \mathbb N$ it is possible to take  a countable closed subset $K_n$
such that $h_{top}(K_n,\mathbb S)=h_{top}(x_n)$ and $x_n$ is its unique limit point in $X$. Moreover,
$K_n \backslash B_{r_n}(x_n)$ is a finite subset and, under such observation,  without loss of generality, we assume $K_n\subset
B_{r_n}(x_n)$. 

Define $K=\{x_0\}\bigcup_{n\in\mathbb N}K_n$. As each $K_n$ is a closed  countable set, $K$ is a countable closed subset of $X$ and the set
of limit points of $K$ in $X$ is just $\{x_0\}\cup\{x_n : n \in\mathbb  N\}$, as $x_n \to x_0$ and $r_n \to 0$ when
$n\to\infty$. 
Finally
$$
h_{top}(K,\mathbb S)\geq h_{top}(K_n,\mathbb S)=h_{top}(x_n)\Rightarrow h_{top}(K,\mathbb S)\geq \lim_{n\to\infty}h_{top}(x_n)=h_{top}(X,\mathbb S),
$$
which ends the proof by taking $K$ as the desired set, since $h_{top}(K,\mathbb S)\leq h_{top}(X,\mathbb S)$.

\section{Examples}
Our first example consider a free semigroup action where the generating set is given by expanding maps.
\begin{example}
We say that a
$C^1$-local diffeomorphism $f: M \to M$ on a compact Riemannian manifold is an \emph{expanding map} if there are constants $C>0$ and
$0<\lambda<1$ such that $\|(Df^n(x))^{-1}\| \le C \lambda^n$ for every $n\ge 1$ and $x\in X$. In \cite{RoVa1} the authors proved that if  $G_1=\{g_1, g_2, \dots ,g_k\}$ is a finite set of expanding maps acting on $M$ and $G$ is the generated
semigroup by $G_1$ then every point $x\in M$ is a full entropy point for   the free semigroup action $\mathbb S$.
\end{example}

The next example shows that it is possible to get a free semigroup action where the fixed generating set is not given by expanding maps, but the set of full entropy points is still the hole phase space.
\begin{example}
For any $\beta>0$, consider the interval map $f_\beta : [0,1] \to [0,1]$ given by
$$
f_\beta(x)
	=\begin{cases}
	\begin{array}{ll}
	x ( 1+ (2x)^\beta ) &, \text{if}\; x\in [0,\frac12] \\
	2x-1  & ,\text{if}\; x\in (\frac12,1]
	\end{array}
	\end{cases}
$$
also known as Maneville-Pomeau map. 
 Although $f_\beta$ is not continuous it induces 
a continuous and topologically mixing circle map $\tilde f_\beta$ taking $\mathbb S^1=[0,1]/\sim$ with the 
identification $0\sim 1$.  
Let $G$ be the semigroup generated by $G_1=\{id, \tilde f_\beta,R_\alpha\}$ where $R_\alpha$ is 
the rotation of angle $\alpha$. Clearly no element of $G_1$ is an expanding map. Again, by \cite{RoVa1}
we have that  every $x\in\mathbb S^1$ is a full entropy point.
\end{example}
In the two previous examples we have that the set of entropy points is the hole phase space. In the next we present a semigroup action where the set of entropy points  is different of the ambient space.
\begin{example}
Let $X_1$ and $X_2$ be compact metric spaces. For $i\in\{1,2\}$ take  $f_i:X_i\to X_i$ and $g_i:X_i\to X_i$, $i\in\{1,2\}$, continuous
maps. Then define $X=X_1\cup X_2$ and
$$
f(x)=\left\{\begin{array}{lr}
         f_1(x), & \text{ if }x\in X_1,\\
       f_2(x), & \text{ if }x\in X_2
        \end{array}\right.
        \text{ and }
        g(x)=\left\{\begin{array}{lr}
         g_1(x), & \text{ if }x\in X_1,\\
       g_2(x), & \text{ if }x\in X_2
        \end{array}\right..
$$
If $f_2=g_2=id_{X_2}$ we have that the topological entropy of the semigroup action generated by $G_1=\{id_X,f,g\}$ coincides 
with the topological entropy of the semigroup generated by $H_1=\{id_{X_1},f_1,g_1\}$. In particular, $E_p(X,\mathbb S)\subset X_1$.

\end{example}
In Proposition \ref{proposition-closed-set} we see that if the topological entropy of a closed subset is positive then this set contains an entropy point. In what follows we are going to show that the converse, in general, is not true. 
\begin{example}
Consider 
$$
C=  \left(\begin{matrix}
                    2 & 1 \\
                    1 & 1
                  \end{matrix}\right)
\text{ and }I=  \left(\begin{matrix}
                    1 & 0 \\
                    0 & 1
                  \end{matrix}\right).
                  $$
Define 
$$
A= \left(\begin{matrix}
                    C & 0 & 0 \\
                    0 & I & 0 \\
                    0 & 0 & 0
                  \end{matrix}\right)
                  \text{ and }B=  \left(\begin{matrix}
                    I & 0 & 0 \\
                    0 & C & 0 \\
                    0 & 0 & 0
                    \end{matrix}\right),
$$
matrices in $M_5(\mathbb R)$. Then we have that $AB=BA$. Theses matrices induces  non transitive linear endomorphisms on the torus $\mathbb T^5=\mathbb R^5\slash\mathbb Z^5$. 
Moreover, since for any $x=(x_1,x_2,x_3,x_4,x_5)\in\mathbb T^5$ and $m,n\in\mathbb N$
$$
A^nB^m(x)=(C^n(x_1,x_2),C^m(x_3,x_4),0),
$$
the action given by the semigroup generated by $G_1=\{id_{\mathbb T^5},A,B\}$ does not admit a point with dense orbit. Another important consequence of the last inequality is that  $h_{top}(\pi(\{(0,0,0,0,x):x\in\mathbb R\},\mathbb S)=0$, where $\pi:\mathbb R^5\to\mathbb R^5\slash\mathbb Z^5$ is the canonical projection. By the other hand, given $z\in\pi(\{(0,0,0,0,x):x\in\mathbb R\}$, we have that $h_{top}(z)>0$, i.e., $z\in E_p(\mathbb T^5,\mathbb S)$.

\end{example}

\subsection*{Acknowledgements}
We thank Paulo Varandas for the valuable discussions that helped to improve the quality of this work.

\end{document}